\documentclass{article}
\pdfoutput=1

\usepackage[title]{appendix}
\usepackage[left=3cm,right=2cm,
    top=2cm,bottom=2cm]{geometry}
\usepackage{amsthm, amssymb,amsmath}
\usepackage{cite}
\usepackage{url}\usepackage{hyperref}
\usepackage{authblk}

\newtheorem{theorem}{Theorem}
\newtheorem{corollary}[theorem]{Corollary}
\newtheorem{proposition}[theorem]{Proposition}
\newtheorem{remark}[theorem]{Remark}
\newtheorem{lemma}[theorem]{Lemma}
\newtheorem{example}[theorem]{Example}
\newtheorem{definition}[theorem]{Definition}

\numberwithin{theorem}{section}
\numberwithin{equation}{section}

\begin{document}

	\title{ Partial Differential Equations in Module of Copolynomials of Several Variables over a Commutative Ring.}

\author{S.L.Gefter \thanks{gefter@karazin.ua} }
\author{A.L.Piven'\thanks{aleksei.piven@karazin.ua} }
\affil{Department of Mathematics \& Computer Sciences\\
	 V. N. Karazin Kharkiv National University}

	\date{}
		\maketitle
This article is accepted for publishing in
Journal of Mathematical Physics, Analysis and Geometry		
		\begin{abstract}
			We study the copolynomials of $n$ variables, i.e. $K$-linear mappings from the ring of polynomials $K[x_1,...,x_n]$  into the commutative ring $K$.
 We prove an existence and uniqueness theorem for a linear differential equation of infinite order which can be considered as an algebraic version of the classical  Malgrange-Ehrenpreis theorem for the existence of the fundamental solution of a linear differential operator with constant coefficients. We find the fundamental solutions of  linear differential operators of infinite order and show that  the unique solution of the corresponding inhomogeneous equation  can be represented as a convolution of the fundamental solution of this operator and the right-hand side.
We also prove the existence and uniqueness theorem of the Cauchy problem for some linear differential equations in the module of formal power series with copolynomial coefficients.
			

		\end{abstract}

\section{Introduction}

The Poisson formula $$u(t,x)=\frac{1}{2a\sqrt{\pi t}}\int\limits_{-\infty}^{\infty }  e^{-\frac{y^2}{4a^2t}}Q(x-y)dy$$ for the solution of the Cauchy problem
for the one-dimensional heat equation
$$\frac{\partial u(t,x)}{\partial t}=a^2\frac{\partial ^2 u(t,x)}{\partial x^2},\quad u(0,x)=Q(x)$$
is very interesting. Externally, this formula is ''very transcendental''.
In  other direction, if the initial condition $Q(x)$ is a polynomial of degree  $m$ with
integer coefficients and $a\in\mathbb{Z}$, then the considered Cauchy problem has the unique polynomial solution with integer coefficients
$$u(t,x)=\sum_{k=0}^{[m/2]}a^{2k}\frac{Q^{(2k)}(x)}{k!}t^k.$$
The form of this solution shows that it is defined over the ring $\mathbb{Z}$,
i.e. to find it we need only to add and to multiply
(unlike to the Poisson formula, where coefficient  $a^2$ be in a denominator of  an exponent of power). Thus, in some sense the Poisson formula has an arithmetic origin.
The analogue of the Poisson formula in the case of the non-invertible operator coefficient
$a^2$ and the convergent power series $Q(x)$ was considered in \cite{GV}.
In the present paper
we consider a purely algebraic version of the Poisson formula  (see Example \ref{ex31}) and other similar formulas in the case when the initial condition  $Q$ is a copolynomial over the commutative integral domain $K$, that is a $K$-linear functional on the ring  $K[x_1,...,x_n]$ of polynomials of $n$ variables. General properties of copolynomials of $n$ variables were considered in Section \ref{sec2}. Recently the case $n=1$ was studied in \cite{GP1,GP,HP} partially. In these papers copolynomials were called a formal generalized functions (see also \cite{GS}). Note that given properties of copolynomials connected with the convolution are  consequences of general constructions of the theory of Hopf algebras (see, for example, \cite{Kur,Und}). In Section \ref{sec3}  differential operators of infinite order on the module of copolynomials are studied.
In Section \ref{sec4} with the help of the  Laplace transform a
connection  between copolynomials and formal power
series is established  (see Propositions  \ref{prop61}, \ref{prop62} and
Theorem \ref{th41}).
The main results of the present paper are contained in Sections \ref{sec6} and \ref{sec7}. Theorem \ref{th2}  and Corollary  \ref{cor2} can be considered as an algebraic version of the classical  Malgrange-Ehrenpreis theorem for the existence of the fundamental solution of a linear differential operator with constant coefficients (see, for example,  \cite[Theorem 7.3.10]{Ho1}, \cite[Section 10.2]{Ho2})). Moreover,
it was shown in Theorem \ref{th2} and Corollary \ref{cor52} that  the unique solution of the inhomogeneous equation ${\mathcal F}u=T$ with a linear differential operator ${\mathcal F}$ of infinite order can be represented as a convolution of the fundamental solution of this operator and the right-hand side  $T$ from the module of copolynomials. We note that unlike the classical theory  (see, for example, \cite{Ho2}), the solution of the inhomogeneous equation ${\mathcal F}u=T$,
and in particular, its fundamental solution in the module of copolynomials
are defined uniquely.  In Section \ref{sec7}  the concept of a fundamental solution of the Cauchy problem for the equation $\frac{\partial u}{\partial t}={\mathcal F}u$
is introduced and is studied. The main result of this section is Theorem \ref{th5} which states that under the fulfillment of  additional restrictions on the ring $K$ the Cauchy problem
$\frac{\partial u}{\partial t}=({\mathcal F}u)(t,x),\ u(0,x)=Q(x)$ with a copolynomial $Q(x)$
has a unique solution and furthermore  this solution is a convolution of the fundamental solution of the Cauchy problem and the initial condition.
In Sections  \ref{sec6} and \ref{sec7} we present meaningful examples which illustrate the constructed theory.

Note that differential operators of infinite order on different spaces  were studied in numerous works (see, for example, \cite{Mor,kn,Dub,St1,St,Tkach,Gorod1,Gorod}).
In the classical scalar case, the series with respect to the derivatives of the
$\delta$-function are intensively studied because of their applications to differential and functional-differential
equations and the theory of orthogonal polynomials (see, for example, \cite{EF,EK}).	

\section{Preliminary}\label{sec2}

Let $K$ be an arbitrary commutative integral domain with identity
 and let $K[x_1,...,x_n]$ be a ring of polynomials with coefficients in $K$.

\begin{definition}\label{def1} By a copolynomial over the ring $K$ we mean a $K$-linear functional defined on the ring $K[x_1,...,x_n]$, i.e. a homomorphism from the module $K[x_1,...,x_n]$  into the ring $K$.
 \end{definition}

We denote the module of copolynomials over $K$ by $K[x_1,..,x_n]'$.
Thus  $T\in K[x_1,...,x_n]'$ if and only if $T:K[x_1,...,x_n]\to K$ and
$T$ has the property of $K$-linearity: $T(ap+bq)=aT(p)+bT(q)$
for all $p,q\in K[x_1,...,x_n]$ and $a,b\in K$.
 If $T\in K[x_1,..,x_n]'$ and $p\in K[x_1,...,x_n]$, the for the value of  $T$ on $p$ we use the notation  $(T,p)$.
 We also write a copolynomial $T\in K[x_1,...,x_n]'$ in the form $T(x)$,
 where $x=(x_1,...,x_n)$ is regarded as the argument of polynomials
 $p(x)\in K[x_1,...,x_n]$ subjected to the action of the  $K$-linear mapping $T$.
In this case, the result of action of $T$
upon $p$ can be represented in the form $(T(x),p(x))$.

Let $\mathbb{N}_0$ be the set of nonnegative integers. For a multi-index $\alpha=(\alpha_1,...,\alpha_n)\in\mathbb{N}_0^n$
we put \cite[Chapter 1, \S 1--2]{Mor}
$$D^{\alpha}=\frac{\partial^{|\alpha|}}{\partial x_1^{\alpha_1}\partial x_2^{\alpha_2}\cdots
\partial x_n^{\alpha_n}},\quad |\alpha|=\sum_{j=1}^n\alpha_j,$$
$$x^{\alpha}=x_1^{\alpha_1}x_2^{\alpha_2}\cdots x_n^{\alpha_n},\quad
\alpha !=\alpha_1 !\alpha_2!\cdots \alpha_n!.$$
For multi-indexes $\alpha, \beta\in\mathbb{N}_0^n$ the relation $\alpha\le \beta$ means that $\alpha_j\le \beta_j$ for all $j=1,...,n$. If $\alpha\le \beta$ , then we will use the notation $\left(\begin{array}{c}\beta\\ \alpha\end{array}\right)=
\prod\limits_{j=1}^n\left(\begin{array}{c}\beta _j\\ \alpha _j\end{array}\right)$.

Let
$p(x)=\sum\limits_{|\alpha|\le m}a_{\alpha}x^{\alpha}\in K[x_1,...,x_n]$.
If  $h=(h_1,...,h_n)$, then the polynomial $p(x+h)\in K[x_1,...,x_n][h_1,...,h_n]$
can be represented in the form
$$p(x+h)=\sum_{|\alpha|\le m}p_{\alpha}(x)h^{\alpha},$$ where  $p_{\alpha}(x)\in K[x_1,...,x_n]$. Since in the case of a field with zero characteristic
$p_{\alpha}(x)=\frac{D^{\alpha}p(x)}{\alpha !}$, we also assume that, by definition,
$\frac{D^{\alpha}p(x)}{\alpha !}=p_{\alpha}(x),\quad |\alpha|\le m$
is true for any commutative ring $K$. For $m<|\alpha |$ we assume that  $\frac{D^{\alpha}p(x)}{\alpha !}=0$.

We now introduce the notion of shift for a copolynomial \cite{GS,GP}.
For $T\in K[x_1,...,x_n]'$ and fixed $h=(h_1,...,h_n)\in K^n$
we define a copolynomial
$T(x+h)$ by the formula
$$(T(x+h),p)=(T,p(x-h)),\quad p\in K[x_1,...,x_n].$$
\begin{definition}\rm
The partial derivative  $\frac{\partial T}{\partial x_j}$ of a copolynomial  $T\in K[x_1,...,x_n]'$ with respect to the variable $x_j\ (j=1,...,n)$ is defined as in the classical case by the formula
 \begin{equation}\label{deriv}\left(\frac{\partial T}{\partial x_j},p\right)=-\left(T,\frac{\partial p}{\partial x_j}\right),\quad p\in K[x_1,...,x_n].\end{equation}
\end{definition}
By using \eqref{deriv}, we arrive at the following expression for the  derivative
 $D^{\alpha }T$:
$$(D^{\alpha }T,p)=(-1)^{|\alpha|}(T,D^{\alpha}p),\quad p\in K[x_1,...,x_n].$$

Therefore
$$
(D^{\alpha}T,p)=0,\ \text{where} \ p\in K[x_1,...,x_n]\ \text{and}\ |\alpha|>{\rm deg} p.$$

By virtue of the equality
\begin{equation}\label{dn}
\left(\frac{D^{\alpha}T}{\alpha !},p\right)=(-1)^{|\alpha|}\left(T,\frac{D^{\alpha}p}{\alpha !}\right),\quad p\in K[x_1,...,x_n],
\end{equation}
the copolynomials $\frac{D^{\alpha}T}{\alpha !}$
are well defined for any  $T\in K[x_1,...,x_n]'$ and $\alpha\in\mathbb{N}_0^n$.

\begin{example}\rm The copolynomial $\delta$-function
is given by the formula
$$(\delta,p)=p(0),\quad p\in K[x_1,...,x_n].$$
Therefore
$$(D^{\alpha}\delta,p)=(-1)^{|\alpha|}(\delta,D^{\alpha}p)=(-1)^{|\alpha|}D^{\alpha }p(0),\quad \alpha\in\mathbb{N}_0^n.$$
\end{example}

\begin{example}\label{ex22}\rm
Let  $K=\mathbb {R}$  and let $f:\mathbb{R}^n\rightarrow \mathbb{R}$ be a Lebesgue-integrable function such that
\begin{equation}\label{reg1}
\int\limits_{\mathbb{R}^n} |x^{\alpha}f(x)|dx
< +\infty,\quad \alpha\in \mathbb{N}_0^n
\end{equation}
Then  $f$ generates the regular copolynomial $T_f$:
$$(T_f,p)=\int\limits_{\mathbb{R}^n} p(x)f(x)
dx,\quad p\in \mathbb{R}[x_1,...,x_n].$$
Note that, in this case, unlike the classical theory, all copolynomials are regular \cite[Theorem 7.3.4]{EK}, although a
nonzero function $f$ can generate the zero copolynomial \cite[Example 2.2]{GP}, \cite[Remark 1]{GS}.
Moreover, if $f\in C^1(\mathbb{R}^n)$ and the conditions \eqref{reg1} are satisfied
for $\frac{\partial f}{\partial x_j}\ (j=1,...,n)$, then it can be shown
that
$$\left(\frac{\partial T_f}{\partial x_j},p\right)=\int\limits_{\mathbb{R}^n} p(x)\frac{\partial f}{\partial x_j}
dx,\quad p\in \mathbb{R}[x_1,...,x_n],\ j=1,...,n.$$
\end{example}

\begin{remark}\rm
The notion of a copolynomial differs from that of
a formal distribution used in the theory of vertex operator algebras (see \cite{v1,v2}),
although there are some natural connections between these notions.
\end{remark}

We now consider the problem of convergence in the space $K[x_1,...,x_n]'$.
In the ring K, we consider the discrete
topology. Further, in the module of copolynomials $K[x_1,...,x_n]'$, we consider the topology of pointwise convergence. It is easy to show that the last topology is generated by the following metric:
$$d(T_1,T_2)=\sum_{|\alpha|= 0}^{\infty }\frac{d_0((T_1,x^{\alpha}),(T_2,x^{\alpha}))}{2^{|\alpha|}},$$
where $d_0$ is the discrete metric on $K$.
 The convergence of a sequence $\{T_k\}_{k=1}^{\infty}$ to $T$ in $K[x_1,...,x_n]'$ means that for every polynomial $p\in K[x_1,...,x_n]$ there exists a number $k_0\in\mathbb{N}$ such that
$$(T_k,p)=(T,p),\quad k=k_0,k_0+1,k_0+2,....$$
The series $\sum\limits_{k=0}^{\infty}T_k$ converges in $K[x_1,...,x_n]'$ if a sequence of its partial sums $\sum\limits_{k=0}^{N}T_k$ converges in $K[x_1,...,x_n]'$.

The following lemma shows the possibility of the decomposition of an arbitrary copolynomial in series about the system
$\frac{D^{\alpha}\delta}{\alpha!},\ \alpha\in\mathbb{N}_0^n$.

\begin{lemma}\label{lem21} Let $T\in K[x_1,...,x_n]'$. Then
$$
T=\sum_{|\alpha|=0}^{\infty}(-1)^{|\alpha|}(T,x^{\alpha})\frac{D^{\alpha}\delta}{\alpha!}.$$
\end{lemma}

\begin{proof}
For any multi-index $\beta\in\mathbb{N}_0^n$ we have
$$\sum_{|\alpha|=0}^{\infty}(-1)^{|\alpha|}(T,x^{\alpha})\left(\frac{D^{\alpha}\delta}{\alpha!},x^{\beta}\right)=
(-1)^{|\beta|}(T,x^{\beta})\left(\frac{D^{\beta}\delta}{\beta!},x^{\beta}\right)=(T,x^{\beta}).$$
\end{proof}

\begin{definition}\label{defconv}\rm If $p\in K[x_1,...,x_n]$ and $T\in K[x_1,...,x_n]'$ is a copolynomial, then their \textit{convolution} $T*p$ is
defined naturally as follows:
$$(T*p)(x)=(T(y),p(x-y))=\sum\limits_{|\alpha|\le m}(-1)^{|\alpha|}(T,y^{\alpha})\frac{D^{\alpha}p(x)}{\alpha!},$$
where $m={\rm deg}p$. Thus $T*p$ is a polynomial with coefficients in $K$, i.e. $T*p\in K[x_1,...,x_n]$.
\end{definition}

\begin{remark}\label{rem2}\rm
By Definition \ref{defconv} we have $\delta*p=p$.
\end{remark}

\begin{definition}\label{deftenz}\rm The \textit{tensor product} $T_1\otimes T_2\in K[x_1,...,x_n,y_1,...,y_m]'$ of copolynomials
$T_1\in K[x_1,....,x_n]'$ and $T_2\in K[y_1,....,y_m]'$ is defined by the equality:
$$(T_1\otimes T_2,x^{\alpha}y^{\beta})=(T_1,x^{\alpha})(T_2,y^{\beta}),\quad
\alpha\in\mathbb{N}_0^n,\beta\in\mathbb{N}_0^m.$$
Further with the help of $K$-linearity the result of the action $(T_1\otimes T_2,p)$ of the copolynomial
$T_1\otimes T_2$ to an arbitrary polynomial $p\in K[x_1,...,x_n,y_1,...,y_m]$ is defined.
\end{definition}

\begin{definition}\label{defconv1}\rm Let $T_1,T_2\in K[x_1,...,x_n]'$. We remind the definition of their
 \textit{convolution} \cite{Kur}, \cite[Section 2.1]{Und}. If $p\in K[x_1,...,x_n]$ and
$p(x+y)=\sum\limits_{|\alpha|\le m}\frac{D^{\alpha}p(x)}{\alpha!}y^{\alpha}$,
then
\begin{equation}\label{conv}
(T_1*T_2,p)=(T_1\otimes T_2,p(x+y))=\sum_{|\alpha| \le m}\left(T_1(x),\frac{D^{\alpha}p(x)}{\alpha!}\right)(T_2(y),y^{\alpha}).\end{equation}
Note that $T_1*T_2\in K[x_1,...,x_n]'$.
\end{definition}
The following assertion  establishes  the commutativity and associativity for a convolution of copolynomials.
\begin{proposition}
Let $T_1,T_2,T_3\in K[x_1,...,x_n]'$. Then
$$T_1*T_2=T_2*T_1$$
$$(T_1*T_2)*T_3=T_1*(T_2*T_3).$$
\end{proposition}

\begin{example}\label{ex2}\rm Let $T\in K[x_1,...,x_n]'$. We find the convolution $\delta*T$. For $p\in K[x_1,...,x_n]$, ${\rm deg }p=m$ we obtain
$$(\delta*T,p)=\sum_{|\alpha|\le m}\left(\delta,\frac{D^{\alpha}p(x)}{\alpha!}\right)(T,y^{\alpha})
=\sum_{|\alpha|\le m}\left(\frac{D^{\alpha}p}{\alpha!}\right)(0)(T,y^{\alpha})=(T,p).$$
Hence $\delta*T=T$.
\end{example}

\begin{corollary}\label{cor5}
 The module $K[x_1,...,x_n]'$ under the convolution operation is an associative commutative  algebra with identity  over the ring $K$.
\end{corollary}

The following theorem establishes the property of the continuity for a convolution.

\begin{theorem}\label{th21}
Let $T_k\in K[x_1,...,x_n]',\ k\in\mathbb{N}$ and $T_k\to 0,\ k\to\infty$ in the topology of $K[x_1,...,x_n]'$. Then $T_k*S\to 0,\quad k\to\infty$ in the topology of $K[x_1,...,x_n]'$ for every copolynomial  $S\in K[x_1,...,x_n]'$.
\end{theorem}
\begin{proof}
Indeed, by Formula \eqref{conv},  for every polynomial  $p\in K[x_1,...,x_n]$ of degree $m$ we have
$$(T_k*S,p)=\sum_{|\alpha |\le m}(T_k(y),y^{\alpha})\left(S(x),\frac{D^{\alpha}p(x)}{\alpha!}\right)\to 0,\ k\to \infty$$
\end{proof}
\begin{corollary}\label{cor1}
 Assume that $T_k\in K[x_1,...,x_n]',\ k\in\mathbb{N}$ and the series $\sum\limits_{k=1}^{\infty}T_k$ converges in the topology of $K[x_1,...,x_n]'$. Then for every copolynomial $S\in K[x_1,...,x_n]'$ the series $\sum\limits_{k=1}^{\infty}(T_k*S)$ converges in the same topology and
$$\sum_{k=1}^{\infty}(T_k*S)=\left(\sum_{k=1}^{\infty}T_k\right)*S.$$
\end{corollary}

\section{Linear differential operators of infinite order on the module of copolynomials}\label{sec3}

We now consider the following linear differential operator of infinite order on $K[x_1,...,x_n]'$:
$$
{\mathcal F}=\sum\limits_{|\alpha|=0}^{\infty}a_{\alpha}D^{\alpha},$$
where $a_{\alpha}\in K$.
This operator acts upon a copolynomial $T\in K[x_1,...,x_n]'$ by the following rule: if
$p\in K[x_1,...,x_n]$ and $m={\rm deg} p$, then
$$({\mathcal F}T,p)=\left(\sum\limits_{|\alpha|=0}^{\infty}a_{\alpha}D^{\alpha}T,p\right)=
\sum_{|\alpha|\le m}(-1)^{|\alpha|}a_{\alpha}(T,D^{\alpha}p)=
\sum_{|\alpha|\le m}a_{\alpha}(D^{\alpha }T,p).$$

Thus, the differential operator ${\mathcal F}:K[x_1,...,x_n]'\to K[x_1,...,x_n]'$ is well defined and for any polynomial  $p$ of degree at most $m$ the equality
\begin{equation}\label{FT}
({\mathcal F}T,p)=\sum_{|\alpha|\le m}a_{\alpha}(D^{\alpha }T,p).
\end{equation}
is true.

\begin{lemma}\label{lem}
The differential operator ${\mathcal F}:K[x_1,...,x_n]'\to K[x_1,...,x_n]'$ is a continuous $K$-linear mapping.
\end{lemma}
\begin{proof}
Assume that a sequence of copolynomials $\{T_k\}_{k=0}^{\infty}$ converges to $T$ in $K[x_1,...,x_n]'$.
Then
there exists $k_0=k_0(p)\in\mathbb{N}$ such that the equality
$(T_k,p)=(T,p)$ is true for all $k\ge k_0(p)$.  Let $m={\rm deg }p$
and $s(p)=\max\{k_0(D^{\alpha}p):|\alpha|\le m\}$.
Then, by using \eqref{FT}, we obtain
$$({\mathcal F}T_k,p)=\left(\sum_{|\alpha|\le m}a_{\alpha}D^{\alpha }T_k,p\right)=
\sum_{|\alpha|\le m}a_{\alpha}(-1)^{\alpha}(T_k,D^{\alpha}p)$$ $$=\sum_{|\alpha|\le m}a_{\alpha}(-1)^{\alpha}(T,D^{\alpha}p)
=\left(\sum_{|\alpha|\le m}a_{\alpha}D^{\alpha}T,p\right)=({\mathcal F}T,p),\quad k\ge s(p),$$
i.e., the sequence $\{{\mathcal F}T_{k}\}_{k=0}^{\infty}$ converges to ${\mathcal F}T$. The lemma is proved.
\end{proof}

The following assertion shows that the convolution operation and the differential operator of infinite order
commute. We also show that every such differential operator is a convolution operator.
\begin{theorem}\label{prop1}
Let  $T_1,T_2\in  K[x_1,...,x_n]'$ and let ${\mathcal F}=\sum\limits_{|\alpha|=0}^{\infty}a_{\alpha}D^{\alpha}$ be a differential operator
of infinite order on $K[x_1,...,x_n]'$ with coefficients $a_{\alpha}\in K$.  Then
$$
{\mathcal F}(T_1*T_2)=({\mathcal F}T_1)*T_2.$$
Therefore ${\mathcal F}(T)={\mathcal F}(\delta)*T$ for all $T\in K[x_1,...,x_n]'$.
\end{theorem}

\begin{proof} Let $p\in K[x_1,...,x_n]$, $m={\rm deg} p$
and let $\beta\in \mathbb{N}_0^n$ be an arbitrary multi-index. By the definition of the convolution
$$(D^{\beta}(T_1*T_2),p)=(-1)^{|\beta|}(T_1*T_2,D^{\beta} p)=(-1)^{|\beta |}\sum_{|\alpha|\le m}\left(T_1,\frac{D^{\alpha+\beta}p}{\alpha !}\right)(T_2,y^{\alpha})$$
$$=\sum_{|\alpha |\le m}\left(D^{\beta}T_1,\frac{D^{\alpha}p}{\alpha !}\right)(T_2,y^{\alpha})=((D^{\beta}T_1)*T_2,p).$$
Therefore
\begin{equation}\label{db}
D^{\beta}(T_1*T_2)=(D^{\beta}T_1)*T_2.
\end{equation}
By \eqref{FT}, for every $T\in K[x_1,...,x_n]'$ the series  $\sum\limits_{|\beta|=0}^{\infty}a_{\beta}D^{\beta}T$
converges in the topology of $K[x_1,....,x_n]'$. Therefore, by Corollary \ref{cor1}
and the equality \eqref{db},
\begin{equation}\label{convdif}
({\mathcal F}(T_1*T_2))=\sum_{|\beta|=0}^{\infty}a_{\beta}D^{\beta}(T_1*T_2)=\sum_{|\beta|=0 }^{\infty}a_{\beta}((D^{\beta}T_1)*T_2)=({\mathcal F}T_1)*T_2.\end{equation}
Now, substituting $T_1=\delta$, $T_2=T$ into \eqref{convdif}, we get ${\mathcal F}(T)={\mathcal F}(\delta)*T$. The theorem is proved.
\end{proof}

\section{ The Laplace transform in the module of copolynomials}\label{sec4}

Let $z=(z_1,...,z_n)$ and let  $K\left[\left[z_1,...,z_n,\frac{1}{z_1},...,\frac{1}{z_n}\right]\right]$ be the module of formal Laurent series with coefficients in $K$.
For the multi-index $\alpha=(\alpha_1,...,\alpha_n)\in\mathbb{Z}^n$
we put $z^{\alpha}=z_1^{\alpha_1}z_2^{\alpha_2}\cdots z_n^{\alpha_n}$. For $g\in K\left[\left[z_1,...,z_n,\frac{1}{z_1},...,\frac{1}{z_n}\right]\right]$, $g(z)=\sum\limits_{\alpha\in\mathbb{Z}^n}g_{\alpha}z^{\alpha}$  we naturally  define  the formal residue:
$$Res(g(z))=g_{(-1,...,-1)}.$$

Now we define a Laplace transformation of a copolynomial $T\in K[x_1,...,x_n]'$.

\begin{definition}\rm Let $T\in K[x_1,...,x_n]'$.  Assume that the ring $K$ contains the field of rational numbers  $\mathbb{Q}$. Consider the following formal power series from $K[[z_1,...,z_n]]$:
$$L(T)(z)=\widetilde{T}(z)=
\sum\limits_{|\alpha|=0}^\infty \frac{(T,x^{\alpha})}{\alpha !} z^{\alpha}.$$
A power series  $\widetilde{T}(z)$ will be called the \textit{Laplace transform} of the copolynomial $T$.
\end{definition}
We may write informally as follows:: $\widetilde{T}(z)=(T,e^{<z,x>})$. Obviously the mapping
$L :K[x_1,...,x_n]'\to K[[z_1,...,z_n]]$, $L(T)=\widetilde{T}$ is a continuous isomorphism of  $K$-modules, if we consider
the standart Krull topology on $K[[z_1,...,z_n]]$ \cite[Section 1, \S 3, Section 4]{Gr}, i.e. the topology of coefficient-wise stabilization.

\begin{proposition}\label{prop61} {\rm (the inversion formula or the Parseval identity)}.  Let
 $K\supset \mathbb{Q}$, $T\in K[x_1,...,x_n]'$ and $p(x)=\sum\limits_{|\alpha|\le m}c_{\alpha}x^{\alpha}\in K[x_1,...,x_n]$. Then
$$(T(x),p(x))=Res(\widetilde{T}(z)\widetilde{p}(z)),$$
where
$\widetilde{p}(z)=\sum\limits_{|\alpha|\le m}\frac{\alpha! c_{\alpha}}{z^{\alpha+\iota}}$ is the Laplace transform of the polynomial  $p(x)$.
\end{proposition}
\begin{proof} It is sufficient to consider the case $p(x)=x^{\beta}$ for some multi-index $\beta\in\mathbb{N}_0^n$. We have
$\widetilde{p}(z)=\frac{\beta! }{z^{\beta+\iota}}$. Therefore
$$\widetilde{T}(z)\widetilde{p}(z)=\sum_{|\alpha|=0}^{\infty}\frac{(T,x^{\alpha})}{\alpha !}z^{\alpha}\frac{\beta !}{z^{\beta+\iota}}$$
and
$Res(\widetilde{T}(z)\widetilde{p}(z))=(T,x^{\beta})$.
\end{proof}

The following theorem asserts that the Laplace transform sends the convolution of copolynomials to the product of their Laplace transforms.

\begin{theorem}\label{th41}
Let $T_1,T_2\in K[x_1,...,x_n]'$.  Assume that the ring $K$ contains the field of rational numbers $\mathbb{Q}$. Then
$$(\widetilde{T_1*T_2})(z)=\widetilde{T_1}(z)\widetilde{T_2}(z).$$
\end{theorem}

\begin{proof} Since for any $\alpha\in\mathbb{N}_0^n$
$$(T_1*T_2,x^{\alpha})=\sum_{|\beta|\le|\alpha|}\left(T_1(x),\frac{D^{\beta}x^{\alpha}}{\beta!}\right)(T_2(y),y^{\beta})$$
$$=\sum_{\beta\le\alpha}\left(T_1(x),\left(\begin{array}{c}\alpha\\ \beta\end{array}\right)
x^{\alpha-\beta}\right)(T_2(y),y^{\beta}),$$
we have
$$(\widetilde{T_1*T_2})(z)=
\sum\limits_{|\alpha|=0}^\infty \frac{(T_1*T_2,x^{\alpha})}{\alpha !} z^{\alpha}$$ $$=\sum_{|\alpha|=0}^{\infty}\sum_{\beta\le\alpha}\left(T_1(x),\frac{\left(\begin{array}{c}\alpha\\ \beta\end{array}\right)}{\alpha!}
x^{\alpha-\beta}\right)(T_2(y),y^{\beta})z^{\alpha}$$
$$=\sum_{|\alpha|=0}^{\infty}\sum_{\beta\le\alpha}\left(T_1(x),\frac{x^{\alpha-\beta}}{(\alpha-\beta)!}\right)
\left(T_2(y),\frac{y^{\beta}}{\beta!}\right)z^{\alpha}=
\widetilde{T_1}(z)\widetilde{T_2}(z).$$\end{proof}

\begin{example}\label{ex42}\rm  Suppose that $n=1$, $K$ is an arbitrary commutative integral domain, $a\in K$ and
$$({\mathcal E}_a,p)=\sum_{k=0}^m  a^k p^{(k)}(0),$$
where $p\in K[x]$, $m={\rm deg} p $. Then ${\mathcal E}_a\in K[x]'$ and
$${\mathcal E}_a=\sum_{j=0}^{\infty }{(-1)}^j a^j \delta^{(j)}.$$
If $K=\mathbb{R}$ and $a>0$, then
$$({\mathcal E}_a,p)=\int\limits_{-\infty}^\infty p(x)f_a(x)dx,$$
where $$f_a(x)=\left\{
{{\begin{array}{*{20}c} \frac{1}{a} e^{-\frac{x}{a}},\quad x > 0
\\
 0,\quad x<0. \\
\end{array} }} \right.$$
(see Example 4 in \cite{GS}).
Let $K$ be again a commutative integral domain such that $K\supset\mathbb{Q}$. If $a,b\in K$, then
$\widetilde{{\mathcal E}_a}(z)=\sum\limits_{j=0}^{\infty}a^jz^j$ and
$(a-b)\widetilde{{\mathcal E}_a}(z)\widetilde{{\mathcal E}_b}(z)=a\widetilde{{\mathcal E}_a}(z)-b\widetilde{{\mathcal E}_b}(z)$. With the help of Theorem \ref{th41} we obtain
$L((a-b)({\mathcal E}_a*{\mathcal E}_b))=\widetilde{a{\mathcal E}_a-b{\mathcal E}_b}$
and
\begin{equation}\label{hilb}
(a-b)({\mathcal E}_a*{\mathcal E}_b)=a{\mathcal E}_a-b{\mathcal E}_b.
\end{equation}
Now the convolution equation \eqref{hilb} can be checked for an arbitrary commutative integral domain $K$ (see also Example 1.1 in \cite{G}, where it was considered the similar equation for formal Laurent series).
Indeed by Theorem \ref{prop1} and Corollary \ref{cor1} we obtain
$$(a-b)({\mathcal E}_a*{\mathcal E_b})=(a-b)\left(\sum_{j=0}^{\infty }{(-1)}^j a^j \delta^{(j)}\right)*\left(\sum_{k=0}^{\infty }{(-1)}^k b^k \delta^{(k)}\right)$$ $$=(a-b)\sum_{j=0}^{\infty}\sum_{k=0}^{\infty}(-1)^{j+k}a^jb^k(\delta^{(j)}*\delta^{(k)})=
(a-b)\sum_{j=0}^{\infty}\sum_{k=0}^{\infty}(-1)^{j+k}a^jb^k\delta^{(j+k)}$$
$$=(a-b)\sum_{j=0}^{\infty}\sum_{l=j}^{\infty}(-1)^{l}a^jb^{l-j}\delta^{(l)}=
(a-b)\sum_{l=0}^{\infty}(-1)^{l}\sum_{j=0}^la^jb^{l-j}\delta^{(l)}$$
$$=\sum_{l=0}^{\infty}(-1)^{l}(a^{l+1}-b^{l+1})\delta^{(l)}=a{\mathcal E}_a-b{\mathcal E}_b.$$
\end{example}

Now we established a connection between Laplace transform of the copolynomial
${\mathcal F}T=\sum\limits_{|\alpha|=0}^{\infty}a_{\alpha}D^{\alpha}T$, where
 $T\in K[x_1,...,x_n]'$, and the symbol $\varphi(z)=\sum\limits_{|\alpha|=0}^{\infty}a_{\alpha}z^{\alpha}$
of the differential operator  ${\mathcal F}$.

\begin{proposition}\label{prop62} Let  $K\supset \mathbb{Q}$ and let ${\mathcal F}=\sum\limits_{|\alpha|=0}^{\infty}a_{\alpha}D^{\alpha}$ be a linear differential operator of infinite order on $K[x_1,...,x_n]'$ with coefficients  $a_{\alpha}\in K$. Then for every $T\in K[x_1,...,x_n]'$ the equality
\begin{equation}\label{dl}
\widetilde{{\mathcal F}T}(z)=\varphi(-z)\widetilde{T}(z)\end{equation}
holds.
\end{proposition}
\begin{proof}
By the definition of a Laplace transform for any multi-index $\alpha\in\mathbb{N}_0^n$ we have
$$\widetilde{D^{\alpha}T}(z)=\sum_{|\beta|=0}^\infty \frac{(D^{\alpha}T,x^{\beta})}{\beta !} z^{\beta}=\sum_{\beta\ge \alpha}(-1)^{|\alpha|}\frac{(T,x^{\beta-\alpha})}{(\beta -\alpha)!} z^{\beta}$$ $$=(-z)^{\alpha}\sum_{|\beta|=0}^{\infty} \frac{(T,x^{\beta})}{{\beta }!} z^{\beta}=(-z)^{\alpha}\widetilde{T}(z).$$
Multiplying this equality by $a_{\alpha}$ and summing on all $\alpha\in\mathbb{N}_0^n$, we obtain Equality \eqref{dl}.
\end{proof}

\section{Fundamental solution of a linear differential operator of infinite order}\label{sec6}

Let $T\in K[x_1,...,x_n]'$ be a copolynomial and let ${\mathcal F}=\sum\limits_{|\alpha|=0}^{\infty}a_{\alpha}D^{\alpha}$ be a
linear differential operator of infinite order on $K[x_1,...,x_n]'$ with coefficients $a_{\alpha}\in K$.
Consider the following differential equation
\begin{equation}\label{de}
{\mathcal F}u=T.
\end{equation}
We prove an existence and uniqueness theorem for Equation \eqref{de} and  continuous
dependence for the unique solution of this equation on $T$.
\begin{theorem}\label{th2}
Let $a_0$ be an invertible element of the ring  $K$. Then
the linear differential operator ${\mathcal F}$ of infinite order is bijective and its inverse operator ${\mathcal F}^{-1}$ is a continuous mapping.
Moreover, if $I$ denotes the identity mapping, then
\begin{equation}\label{f1}
{\mathcal F}^{-1}=a_0^{-1}\sum_{k=0}^{\infty}(I-a_0^{-1}{\mathcal F})^k,
\end{equation}
where the series in the right-hand side of \eqref{f1}
converges in the topology of $K[x_1,...,x_n]'$.
In particular,
for any copolynomial $T\in K[x_1,...,x_n]'$ there exists a unique solution
$u\in K[x_1,...,x_n]'$ of Equation \eqref{de}. This solution admits representations
$$u={\mathcal F}^{-1}(T)={\mathcal F}^{-1}(\delta)*T$$
and continuously depends on $T$ in the topology of $K[x_1,...,x_n]'$.
\end{theorem}
\begin{proof}
We have the following representation of the operator ${\mathcal F}$:
\begin{equation}\label{do}
{\mathcal F}=a_0\left(I-\sum_{j=1}^n\frac{\partial}{\partial x_j}{\mathcal G}_j\right),
\end{equation}
 where ${\mathcal G}_j \ (j=1,...,n)$ are some linear differential operators. For every $k\in\mathbb{N}$ we have
$$\left(\sum_{j=1}^n\frac{\partial}{\partial x_j}{\mathcal G}_j\right)^k=\sum_{|\alpha|=k}
\frac{k!}{\alpha!}D^{\alpha}{\mathcal G}_1^{\alpha_1}\cdots {\mathcal G}_n^{\alpha_n}.$$
Now for every copolynomial  $T\in K[x_1,...,x_n]'$ and polynomial  $p\in K[x_1,...,x_n]$ of degree $m$ we have
$$\sum_{k=0}^{\infty}\left(\left(\sum_{j=1}^n\frac{\partial}{\partial x_j}{\mathcal G}_j\right)^kT,p\right)=\left(
\sum_{k=0}^{\infty}\sum_{|\alpha|=k}
\frac{k!}{\alpha!}D^{\alpha}{\mathcal G}_1^{\alpha_1}\cdots {\mathcal G}_n^{\alpha_n}T,p\right)$$
$$
=\sum_{|\alpha |=0}^{\infty}\left(\frac{|\alpha |!}{\alpha!}D^{\alpha}{\mathcal G}_1^{\alpha_1}\cdots {\mathcal G}_n^{\alpha_n}T,p\right)=
\sum_{|\alpha |\le m}(-1)^{|\alpha|}\left(\frac{|\alpha |!}{\alpha!}{\mathcal G}_1^{\alpha_1}\cdots {\mathcal G}_n^{\alpha_n}T,D^{\alpha}p\right).$$
 Therefore the series $\sum\limits_{k=0}^{\infty}\left(\sum\limits_{j=1}^n\frac{\partial}{\partial x_j}{\mathcal G}_j\right)^kT$ converges for any copolynomial  $T\in K[x_1,...,x_n]'$, the operator $I-\sum\limits_{j=1}^n\frac{\partial}{\partial x_j}{\mathcal G}_j$  is bijective and its inverse operator has the form:
$$\left(I-\sum\limits_{j=1}^n\frac{\partial}{\partial x_j}{\mathcal G}_j\right)^{-1}=
\sum\limits_{k=0}^{\infty}\left(\sum\limits_{j=1}^n\frac{\partial}{\partial x_j}{\mathcal G}_j\right)^k.$$
Now  \eqref{do} implies the bijectivity of the operator ${\mathcal F}$ and
\begin{equation}\label{inv}
{\mathcal F}^{-1}=a_0^{-1}\sum\limits_{k=0}^{\infty}\left(\sum\limits_{j=1}^n\frac{\partial}{\partial x_j}{\mathcal G}_j\right)^k=a_0^{-1}\sum_{k=0}^{\infty}(I-a_0^{-1}{\mathcal F})^k.\end{equation}
Therefore the representation   \eqref{f1} is true for the inverse operator ${\mathcal F}^{-1}$.
Hence, the differential equation \eqref{de} has a unique solution
$u\in K[x_1,...,x_n]'$ and, moreover, $u={\mathcal F}^{-1}T$. By the equality \eqref{inv} and Corollary \ref{cor1}
we obtain the following representation for this solution:
$$u={\mathcal F}^{-1}T=a_0^{-1}\sum\limits_{k=0}^{\infty}\left(\sum\limits_{j=1}^n\frac{\partial}{\partial x_j}{\mathcal G}_j\right)^kT$$ $$=a_0^{-1}\sum\limits_{k=0}^{\infty}\left(\sum\limits_{j=1}^n\frac{\partial}{\partial x_j}{\mathcal G}_j\right)^k(\delta*T)=\left(a_0^{-1}\sum\limits_{k=0}^{\infty}\left(\sum\limits_{j=1}^n\frac{\partial}{\partial x_j}{\mathcal G}_j\right)^k\delta\right)*T=
{\mathcal F}^{-1}(\delta)*T$$
(see also Example \ref{ex2}).

Now the continuity of the operator ${\mathcal F}^{-1}$ follows from the continuity of the convolution
(see Theorem \ref{th21}).
The theorem is proved.
\end{proof}
\begin{corollary}\label{cor2} Let $a_0$ be an invertible element of the ring $K$.
Then the differential equation  ${\mathcal F}u=\delta$
has the unique solution
\begin{equation}\label{fund}
{\mathcal E}={\mathcal F}^{-1}\delta.
\end{equation}
\end{corollary}
\begin{definition}\label{def2}
The copolynomial  defined by the formula \eqref{fund}
is called the \textit{fundamental solution of the linear differential operator  ${\mathcal F}$}.
\end{definition}
Theorem  \ref{th2} implies the following assertion.
\begin{corollary}\label{cor52}
Let $a_0$ be an invertible element of the ring $K$.
Then the unique solution of Equation \eqref{de} is the convolution of the fundamental solution ${\mathcal E}$
and the copolynomial $T$: $\ u={\mathcal E}*T$.
\end{corollary}

\begin{remark}\rm In all previous results it was supposed that $a_0$ is an invertible element of the ring $K$.
We note that this condition is necessary
 for the existence of the fundamental solution of the differential operator ${\mathcal F}$.
Indeed, if  ${\mathcal E}$ is the fundamental solution of this operator,
then applying the left-hand and right-hand sides of Equation ${\mathcal F}{\mathcal E}=\delta$ to 1, we get that
$a_0({\mathcal E},1)=1$, i.e. $a_0$ is an invertible element of the ring $K$.
\end{remark}

\begin{remark}\rm Let $a_0$ be an invertible element of $K$. Then the differential operator ${\mathcal F}: K[x_1,...,x_n]\to K[x_1,...,x_n]$ is bijective and for any polynomial $p\in K[x_1,...,x_n]$ the differential equation
$${\mathcal Fu}=p$$
has a unique solution $u\in K[x_1,...,x_n]$, moreover ${\rm deg }u\le {\rm deg}p$.  Furthermore,  this solution is a convolution of the
fundamental solution ${\mathcal E}$ of the differential operator ${\mathcal F}$ and the polynomial $p$: $\ u={\mathcal E}*p$.
The proof of this assertion is similar to the proof of Theorem \ref{th2}.
\end{remark}

\begin{example}\rm
Let $n=1$ and $a\in K$.
Consider the linear differential operator ${\mathcal F}=a\frac{d}{dx}+I$ on $K[x]'$. By Corollary \ref{cor2}
the operator ${\mathcal F}$ has the fundamental solution
$${\mathcal E}_a={\mathcal F}^{-1}\delta=\left(I+a\frac{d}{dx}\right)^{-1}\delta=\sum\limits_{j=0}^{\infty }{(-1)}^j a^j \delta^{(j)},$$ i.e.
$a\frac{d{\mathcal E}_a}{dx}+{\mathcal E}_a=\delta$ (see also Example \ref{ex42}). If $K=\mathbb{R}$ and $a>0$, then we obtain that the fundamental solution ${\mathcal E}_a$ regarded as a regular copolynomial  coincides with the classical fundamental solution $\frac{1}{a}\ \theta(x)e^{-x/a}$ of the differential operator ${\mathcal F}$, where $\theta (x)$ is the Heaviside function (see also Examples 4 and 5 in \cite{GS}).
\end{example}

\begin{example}\rm  The linear differential operator of infinite order
${\mathcal F}=\sum\limits_{k=0}^{\infty}\left(\sum\limits_{j=1}^n\frac{\partial}{\partial x_j}\right)^k$ has the inverse operator
$I-\sum\limits_{j=1}^n\frac{\partial}{\partial x_j}$. Therefore ${\mathcal E}=\delta-\sum\limits_{j=1}^n\frac{\partial \delta}{\partial x_j}$ is the fundamental solution of the  operator
${\mathcal F}$. By Theorem \ref{th2}
for every copolynomial  $T\in K[x_1,...,x_n]'$
the differential equation of infinite order
$$\sum_{k=0}^{\infty}\left(\sum\limits_{j=1}^n\frac{\partial}{\partial x_j}\right)^ku=T$$
has a unique solution $u={\mathcal F}^{-1}T=T-\sum\limits_{j=1}^n\frac{\partial T}{\partial x_j}$.
\end{example}

\begin{example}\label{ex1}\rm Let $c$ be an invertible element of the ring $K$.
We consider in the module
$K[x_1,x_2,x_3]'$ the Helmholtz equation
\begin{equation}\label{helm}
\triangle {\mathcal E}+c{\mathcal E}=\delta,
\end{equation}  where $\triangle=\frac{\partial^2}{\partial x_1^2}
+\frac{\partial^2}{\partial x_2^2}+\frac{\partial^2}{\partial x_3^2}$ is the Laplace operator.
By  Theorem \ref{th2}, Equation \eqref{helm} has the unique solution (see Formula \eqref{f1}):
\begin{equation}\label{solhelm}
{\mathcal E}=(cI+\triangle)^{-1}\delta=\sum_{k=0}^{\infty}(-1)^kc^{-k-1}\triangle^k\delta.
\end{equation}
For any $\beta=(\beta_1,\beta_2,\beta_3)\in\mathbb{N}_0^3$ and $k\in\mathbb{N}_0$ we have:
$$\triangle^kx^{\beta}=\sum_{|\alpha|=k}\frac{k!}{\alpha!}D^{2\alpha}x^{\beta},\quad x^{\beta}=x_1^{\beta_1}x_2^{\beta_2}x_3^{\beta_3}.$$
Therefore
$$(\triangle^{|\alpha|}\delta,x^{\beta})=(\delta,\triangle^{|\alpha|}x^{\beta})=
\left\{\begin{array}{ccc}\frac{|\alpha|!(2\alpha)!}{\alpha!},\quad \beta=2\alpha,\\
0,\quad  \beta\not=2\alpha.\end{array}\right.$$
Substituting this expression into the formula \eqref{solhelm}, we obtain
$$({\mathcal E},x^{\beta}) = \left\{\begin{array}{ccc}(-1)^{|\alpha|}\frac{|\alpha|!(2\alpha)!}{\alpha!c^{|\alpha|+1}},\quad \beta=2\alpha,\\
0,\quad  \beta\not=2\alpha.\end{array}\right.$$
This formula gives the fundamental solution of the Helmholtz operator $\triangle +cI$. In the case $K=\mathbb{R}$ and $c>0$ this solution
is connected with classical fundamental solutions
$-\frac{1}{4\pi}\frac{e^{\pm i\sqrt{c}|x|}}{|x|},\ (|x|=\sqrt{x_1^2+x_2^2+x_3^2})$
of the Helmholtz operator by equalities
\begin{equation}\label{int}
({\mathcal E},x^{\beta}) = \lim\limits_{b\to+0}
\int\limits_{\mathbb{R}^3}e^{-b|x|}\left(-\frac{\cos(\sqrt{c}|x|)}{4\pi |x| }\right)x^{\beta}dx,\quad \beta\in\mathbb{N}_0^3,\end{equation}
where the integral in the right-hand side \eqref{int} is calculated with the help of the converting to the spherical coordinates.
\end{example}
\begin{example}\label{ex33}\rm Let $a,c\in K$ and let $c$ be an invertible element of the ring $K$. We find the fundamental solution of the linear differential operator
${\mathcal F}=\frac{\partial }{\partial t}-a\frac{\partial^2}{\partial x^2}+cI$.
We have $${\mathcal F}=c\left(I-\left(ac^{-1}\frac{\partial^2}{\partial x^2}-c^{-1}\frac{\partial }{\partial t}\right)\right).$$
Therefore taking into account  \eqref{f1} and \eqref{fund}, we obtain the following expression for the fundamental solution of the operator ${\mathcal F}$:
$${\mathcal E}={\mathcal F}^{-1}\delta = \sum_{k=0}^{\infty}c^{-k-1}\left(a\frac{\partial^2}{\partial x^2}-\frac{\partial }{\partial t}\right)^k\delta$$ $$=\sum_{k=0}^{\infty}c^{-k-1}\sum_{j=0}^k
\left(\begin{array}{c}k\\ j\end{array}\right)(-1)^ja^{k-j}\frac{\partial^{2k-j}\delta}{\partial t^j\partial x^{2k-2j}}.
$$
This implies that for every $s,l\in\mathbb{N}_0$
$$({\mathcal E},x^mt^l)=\left\{\begin{array}{ccc}\frac{(2s)! (l+s)!}{s!}a^sc^{-l-s-1},\quad m=2s, \\
0,\quad m=2s+1.\end{array}\right.$$

This result also can be obtained with the help of the Laplace transform.
Assume that the ring $K$ contains the field of rational numbers $\mathbb{Q}$.
We apply the Laplace transform to both sides of Equation $$\frac{\partial {\mathcal E}}{\partial t}-a\frac{\partial^2{\mathcal E}}{\partial x^2}+c{\mathcal E}=\delta(t,x).
$$ By Proposition \ref{prop62}
we obtain
 $$(c-az_2^2-z_1)\widetilde{{\mathcal E}}(z_1,z_2)=1.$$
Since $c$ is an invertible element of the ring  $K$, the polynomial $c-az_2^2-z_1$ is an invertible element of the ring  $K[[z_1,z_2]]$. Then
 $$\widetilde{{\mathcal E}}(z_1,z_2)=\frac{1}{c-az_2^2-z_1}=\sum_{k=0}^{\infty}c^{-k-1}\sum_{j=0}^k\left(\begin{array}{c}k\\ j\end{array}\right)a^{k-j}z_1^jz_2^{2k-2j}.$$
Let $p(t,x)=x^mt^l$. Then $\widetilde{p}(z_1,z_2)=\frac{m!l!}{z_1^{l+1}z_2^{m+1}}$ and
 $$\widetilde{{\mathcal E}}(z_1,z_2)\widetilde{p}(z_1,z_2)= \sum_{k=0}^{\infty}c^{-k-1}\sum_{j=0}^k\left(\begin{array}{c}k\\ j \end{array}\right)a^{k-j}m!l!
z_1^{j-l-1}z_2^{2k-2j-m-1}.$$
Now by Proposition  \ref{prop61}, we obtain
$$({\mathcal E}(t,x),p(t,x))=Res(\widetilde{{\mathcal E}}(z_1,z_2)\widetilde{p}(z_1,z_2))=\left\{\begin{array}{ccc}\frac{(2s)! (l+s)!}{s!}a^sc^{-l-s-1},\quad m=2s, \\
0,\quad m=2s+1.\end{array}\right.$$

Now let $K=\mathbb{R}$, $a>0$ and $c>0$. Note that
$$\int\limits_0^{\infty }dt\int\limits_{-\infty}^{\infty } t^le^{-ct}x^m \frac{e^{-\frac{x^2}{4at}}}{\sqrt{4\pi at}}dx
=\left\{\begin{array}{ccc}\frac{(2s)! (l+s)!}{s!}a^sc^{-l-s-1},\quad m=2s, \\
0,\quad m=2s+1.\end{array}\right.$$
Therefore in the space $\mathbb{R}[t,x]'$ the fundamental solution of the differential operator ${\mathcal F}$, regarded as a
regular copolynomial, coincide with the function
$\frac{\theta (t)e^{-ct}}{\sqrt{4\pi at}}e^{-\frac{x^2}{4at}}$.
\end{example}
\begin{example}\rm
We find the fundamental solution of the differential operator
${\mathcal F}=\frac{\partial ^2}{\partial x\partial t}+\frac{\partial }{\partial x}-\frac{\partial }{\partial t}-I$.
By Definition \ref{def2} it is a solution of the differential equation
$$
\frac{\partial ^2{\mathcal E}}{\partial x\partial t}+\frac{\partial {\mathcal E}}{\partial x}-\frac{\partial {\mathcal E}}{\partial t}-{\mathcal E}=\delta.$$
Then the sequence  $C_{sl}=({\mathcal E},x^st^l)\ (l,s\in\mathbb{N}_0)$ is a solution of the following problem for the difference equation
$$C_{sl}=slC_{s-1,l-1}-sC_{s-1,l}+lC_{s,l-1},\quad s,l\in\mathbb{N};$$
$$\quad C_{s0}=(-1)^{s+1}s!,\quad C_{0l}=-l!,\quad s,l\in\mathbb{N}_0.$$
This problem has a unique solution
$$C_{sl}=(-1)^{s+1}l!s!,\quad s,l=0,1,2,....$$

We note that
$$-\int_0^{\infty}dt\int_{-\infty}^0e^{-t+x}t^lx^sdx=(-1)^{s+1}l!s!.$$
Therefore, in the space $\mathbb{R}[t,x]'$
the fundamental solution of the differential operator ${\mathcal F}$, regarded as a
regular copolynomial, coincide with the function
$-\theta(t)\theta(-x)e^{x-t}$.
\end{example}

\begin{example}\rm
We find the fundamental solution of the transport operator
${\mathcal F}=\frac{\partial }{\partial t}+\sum\limits_{i=1}^ns_i\frac{\partial}{\partial x_i}+I$, where $s_i\in K$.
Therefore, taking into account formulas \eqref{f1} and \eqref{fund} we obtain the expression for the fundamental solution of the differential operator ${\mathcal F}$:
$${\mathcal E}(t,x)=({\mathcal F}^{-1}\delta) (t,x)= \sum_{k=0}^{\infty}(-1)^k\left(\frac{\partial }{\partial t}+\sum_{i=1}^ns_i\frac{\partial}{\partial x_i}\right)^k\delta (t,x)$$ $$=\sum_{k=0}^{\infty}(-1)^k\sum_{j=0}^k
\left(\begin{array}{c}k\\ j\end{array}\right)\frac{\partial^{j}}{\partial t^j}\left(\sum_{i=1}^ns_i\frac{\partial}{\partial x_i}\right)^{k-j}\delta (t,x)=
$$
$$=\sum_{k=0}^{\infty}(-1)^k\sum_{j=0}^k
\left(\begin{array}{c}k\\ j\end{array}\right)\frac{\partial^{j}}{\partial t^j}\sum_{|\alpha|=k-j}
\frac{|\alpha |!}{\alpha !}s^{\alpha}D^{\alpha}\delta (t,x),\quad s=(s_1,...,s_n).$$
Then, for every  $l\in\mathbb{N}_0$ and $\beta\in\mathbb{N}_0^n$ we have:
$$({\mathcal E},t^lx^{\beta})=\sum_{k=0}^{\infty}(-1)^k\sum_{j=0}^k
\left(\begin{array}{c}k\\ j\end{array}\right)\sum_{|\alpha|=k-j}
\frac{|\alpha |!}{\alpha !}s^{\alpha}(-1)^{|\alpha|+j}\frac{\partial^{j}}{\partial t^j}D^{\alpha}(t^lx^{\beta})|_{t=0,x=0}  $$ $$=
s^{\beta}(|\beta|+l)!.$$
Now let  $K=\mathbb{R}$. Note that
$$\int_0^{\infty}e^{-t}t^l(\delta(x-ts),x^{\beta})dt=s^{\beta}(|\beta|+l)!,\quad l\in \mathbb{N}_0,\ \beta\in\mathbb{N}_0^n,\ s\in\mathbb{R}^n.$$

Thus a connection between the fundamental solution
of the transport operator
and the classical fundamental solution $\theta(t)e^{-t}\delta(x-ts)$ of this operator is established:
$$({\mathcal E},t^lx^{\beta})=\int_0^{\infty}e^{-t}t^l(\delta(x-ts),x^{\beta})dt,\quad l\in \mathbb{N}_0,\beta\in\mathbb{N}_0^n.$$
\end{example}

\begin{example}\rm
Now we consider the  $m$-th order ordinary linear differential  equation in the module $K[x]'$ of copolynomials of one variable
\begin{equation}\label{dem}
\sum_{j=0}^ma_j\frac{d^ju}{dx^j}=T,
\end{equation}
 where $a_j\in K, j=0,...,m$, $a_0\not=0,\ a_m\not=0$ and $T,u\in K[x]'$ are known
and unknown copolynomials of one variable respectively.
This equation is a particular case of Equation \eqref{de}
with the differential operator
${\mathcal F}=\sum\limits_{j=0}^ma_j\frac{d^j}{dx^j}$.
Assume that $a_0$ is an invertible element of the ring $K$.
Then by Theorem  \ref{th2}, Equation \eqref{dem} has a unique solution. This solution has the form
\begin{equation}\label{sol1}
u=a_0^{-1}\sum_{k=0}^{\infty}(I-a_0^{-1}{\mathcal F})^kT.
\end{equation}
Now for every $k\in\mathbb{N}_0$ we have
$$(I-a_0^{-1}{\mathcal F})^k=(-1)^k\frac{d^k}{dx^k}\left(\sum_{j=1}^ma_0^{-1}a_j\frac{d^{j-1}}{dx^{j-1}}\right)^k
$$ $$=(-1)^k\sum_{|\gamma|=k}\frac{k!}{\gamma!}a_0^{-k}a_1^{\gamma_1}\cdot\ldots\cdot a_m^{\gamma_m}
\frac{d^{k+\sum\limits_{j=1}^m(j-1)\gamma_j}}{dx^{k+\sum\limits_{j=1}^m(j-1)\gamma_j}}.$$
Substituting this expression into \eqref{sol1}, we obtain the following representation for the unique solution of Equation \eqref{dem}:
$$
u(x)=\sum_{k=0}^{\infty}\sum_{|\gamma|=k}(-1)^k\frac{k!}{\gamma!}a_0^{-k-1}a_1^{\gamma_1}\cdot\ldots\cdot a_m^{\gamma_m}
T^{\left(k+\sum\limits_{j=1}^m(j-1)\gamma_j\right)}(x).$$
(see \cite[Section 4]{GGG}, where a similar formula has been obtained in the other situation).
In particular, the first order equation $a_1u'(x)+a_0u(x)=T(x)$
has the unique solution
$$u(x)=\sum_{k=0}^{\infty}(-1)^ka_0^{-k-1}a_1^kT^{(k)}(x),$$
and the second order equation $a_2u''(x)+a_1u'(x)+a_0u(x)=T(x)$
has the unique solution
$$u(x)=\sum_{k=0}^{\infty}\sum_{j=0}^k(-1)^k\left(\begin{array}{c}k\\ j\end{array}\right)a_0^{-k-1}a_1^{k-j}a_2^j
T^{(k+j)}(x)$$ $$=\sum_{j=0}^{\infty}\sum_{k=j}^{\infty}(-1)^k\left(\begin{array}{c}k\\ j\end{array}\right)a_0^{-k-1}a_1^{k-j}a_2^j
T^{(k+j)}(x)$$ $$=\sum_{j=0}^{\infty}\sum_{k=0}^{\infty}(-1)^{k+j}\left(\begin{array}{c}k+j\\ j\end{array}\right)a_0^{-j-k-1}a_1^ka_2^j
T^{(k+2j)}(x)$$ $$=\sum_{s=0}^{\infty}\left(\sum_{j=0}^{\left[\frac{s}{2}\right]}
(-1)^{s-j}\left(\begin{array}{c}s-j\\ j\end{array}\right)a_0^{j-s-1}a_1^{s-2j}a_2^j\right)
T^{(s)}(x)$$
(see Formula (4.10) in \cite{GP}).

\end{example}

\section{Fundamental solution of the Cauchy problem for a linear differential equation in the module of copolynomials}\label{sec7}

\subsection{Formal power series over the module of copolynomials.}\label{sub71}

The module of formal power series of the form $u(t,x)=\sum\limits_{k=0}^{\infty}u_k(x)t^k$ with coefficients $u_k(x)\in K[x_1,...,x_n]'$
will be denoted by $K[x_1,...,x_n]'[[t]]$.

The partial derivative with respect to $t$ of the series $u(t,x)\in K[x_1,....,x_n]'[[t]]$
is defined by the formula
$$\frac{\partial u}{\partial t}=\sum_{k=1}^{\infty}ku_k(x)t^{k-1}.$$
The partial derivatives  $D^{\alpha}$ with respect to variables  $x_1,...,x_n$ of the series $u(t,x)\in K[x_1,...,x_n]'[[t]]$ is defined as follows:
$$D^{\alpha}u(t,x)=\sum_{k=0}^{\infty}(D^{\alpha}u_k)(x)t^k.$$
The action of the  $K$-linear operator ${\mathcal A}:K[x_1,....,x_n]'\to K[x_1,....,x_n]'$ on a formal
power series $u(t,x)=\sum\limits_{k=0}^{\infty}u_k(x)t^k\in K[x_1,...,x_n]'[[t]]$ is defined coefficient-wisely:
$$
({\mathcal A}u)(t,x)=\sum_{k=0}^{\infty}({\mathcal A}u_k)(x)t^k.$$
Obviously, that if ${\mathcal A}$ is an invertible $K$-linear operator on the module $K[x_1,....,x_n]'$, then its extension on the module  $K[x_1,...,x_n]'[[t]]$ is also invertible.

We denote by $(u(t,x),p(x))$ the action of  $u(t,x)\in K[x_1,...,x_n]'[[t]]$ on $p(x)\in K[x_1,...,x_n]$, which is defined coefficient-wisely:
$$(u(t,x),p(x))=\sum\limits_{k=0}^{\infty}(u_k(x),p(x))t^k.$$
Thus, $(u(t,x),p(x))\in K[[t]]$.
\begin{definition}\label{def25}\rm
Let
$u(t,x)=\sum\limits_{k=0}^{\infty}u_k(x)t^k\in K[x_1,...,x_n]'[[t]]$.
The \textit{convolution } of a copolynomial $T\in K[x_1,...,x_n]'$ and a formal power series $u(t,x)$ is also defined coefficient-wisely:
$$
(T*u)(t,x)=\sum_{k=0}^{\infty}(T*u_k(x))t^k,$$
Thus, $(T*u)(t,x)\in K[x_1,...,x_n]'[[t]]$.
\end{definition}

\subsection{The Cauchy problem for a linear partial differential equation in the module of copolynomials.}\label{sub72}

Let ${\mathcal F}=\sum\limits_{|\alpha|=0}^{\infty}a_{\alpha}D^{\alpha}$ be a linear
differential equation of infinite order on $K[x_1,...,x_n]'$ with coefficients $a_{\alpha}\in K$.
In the module $K[x_1,...,x_n]'[[t]]$, we consider
the Cauchy problem
\begin{equation}\label{1}
\frac{\partial u(t,x)}{\partial t}=({\mathcal F}u)(t,x),
\end{equation}
\begin{equation}\label{2}
u(0,x)=Q(x)\in K[x_1,...,x_n]'.
\end{equation}
The following example shows that if $a_0$ is invertible, then this Cauchy problem may has no solutions.
\begin{example}\label{ex3}\rm Let $K=\mathbb{Z}$, ${\mathcal F}=I$ and $Q(x)=\delta(x)$. Then the Cauchy problem \eqref{1}, \eqref{2} is written in the form:
\begin{equation}\label{d1}
\frac{\partial u(t,x)}{\partial t}=u(t,x),
\end{equation}
\begin{equation}\label{d2}
u(0,x)=\delta(x).
\end{equation}
Any solution of this problem can be represented in the form of a formal power series
$u(t,x)=\sum\limits_{k=0}^{\infty}u_k(x)t^k$ with coefficients $u_k(x)\in \mathbb{Z}[x_1,...,x_n]'$. Substituting this representation into \eqref{d1}, \eqref{d2}, we get
$$u_0(x)=\delta(x),\quad (k+1)u_{k+1}(x)=u_k(x),\quad k=0,1,2,....$$
This implies $2(u_2,1)=1$, which contradicts the condition $(u_2,1)\in\mathbb{Z}$.
\end{example}
The following theorem shows that in the case $a_0=0$  the Cauchy problem  \eqref{1}, \eqref{2} has a unique solution.
\begin{theorem}\label{th4}
Let $a_0=0$ and let the ring  $K$ be of characteristic 0. Then
for any copolynomial $Q\in K[x_1,...,x_n]'$
the formal power series
\begin{equation}\label{sol}
u(t,x)=\sum_{k=0}^{\infty}\frac{({\mathcal F}^kQ)(x)}{k!}t^k
\end{equation}
 is well defined and is a unique solution of the Cauchy problem \eqref{1}, \eqref{2}. Furthermore, for every $t\in K$ the series \eqref{sol} converges in the topology of the module $K[x_1,...,x_n]'$.
\end{theorem}
\begin{proof}
We shows first that $K[x_1,...,x_n]'$ is a torsion-free $\mathbb{Z}$-module \cite[Section VII, \S 2]{burb}. Suppose that an element  $T\in K[x_1,...,x_n]'$  satisfies the equality $kT=0$ for some natural $k$. Then $(kT,p)=k(T,p)=0$ for every polynomial $p\in K[x_1,...,x_n]$.
Since the integral domain $K$ is of characteristic 0, we have
$(T,p)=0$, i.e. $T=0$.

Now we prove that the formal power series \eqref{sol} is well defined and
is a unique solution of the Cauchy problem \eqref{1}, \eqref{2}.
Since $a_0=0$, we obtain the following representation for the operator ${\mathcal F}$:
$${\mathcal F}=\sum_{j=1}^n\frac{\partial}{\partial x_j}{\mathcal G}_j,$$
where ${\mathcal G}_j \ (j=1,...,n)$ are some differential operators.
Therefore
$${\mathcal F}^k=k!\sum_{|\alpha|=k}
\frac{D^{\alpha}{\mathcal G}_1^{\alpha_1}\cdots {\mathcal G}_n^{\alpha_n}}{\alpha!},\quad k\in\mathbb{N}.$$
Since copolynomials
$\frac{D^{\alpha}{\mathcal G}_1^{\alpha_1}\cdots {\mathcal G}_n^{\alpha_n}Q}{\alpha !} $ are well-defined (see \eqref{dn}), the element ${\mathcal F}^kQ$ in the module $K[x_1,...,x_n]'$ is divided by  $k!$. Since the module  $K[x_1,...,x_n]'$ is a torsion-free  $\mathbb{Z}$-module, we get
\begin{equation}\label{fk}
\frac{{\mathcal F}^kQ}{k!}=\sum_{|\alpha|=k}
\frac{D^{\alpha}{\mathcal G}_1^{\alpha_1}\cdots {\mathcal G}_n^{\alpha_n}Q}{\alpha!},\quad k\in\mathbb{N}.\end{equation}
 By Theorem 2.3 \cite{GP1} the series \eqref{sol} is well-defined, the Cauchy problem  \eqref{1}, \eqref{2} has a unique solution and this solution has the form \eqref{sol}. Now we show that for any $t\in K$ the series \eqref{sol} converges in $K[x_1,...,x_n]'$.
 We
consider the partial sums $u_N(t,x)=\sum\limits_{k=0}^N\frac{({\mathcal F}^kQ)(x)}{k!}t^k$ of this series and show that they are stabilized on every polynomial.
By equalities \eqref{dn} and \eqref{fk} for any $p\in K[x_1,...,x_n]$ we have
$$(u_N(t,x),p(x))=
\sum_{k=0}^N\sum_{|\alpha|=k}
\left(\frac{D^{\alpha}{\mathcal G}_1^{\alpha_1}\cdots {\mathcal G}_n^{\alpha_n}Q}{\alpha!},p\right)t^k$$
$$=\sum_{k=0}^N\sum_{|\alpha|=k}(-1)^{|\alpha|}
\left({\mathcal G}_1^{\alpha_1}\cdots {\mathcal G}_n^{\alpha_n}Q,\frac{D^{\alpha}p}{\alpha !}\right)t^k$$
$$=
\sum_{k=0}^m\sum_{|\alpha|=k}(-1)^{|\alpha|}
\left({\mathcal G}_1^{\alpha_1}\cdots {\mathcal G}_n^{\alpha_n}Q,\frac{D^{\alpha}p}{\alpha !}\right)t^k=
\sum_{k=0}^m\left(\frac{({\mathcal F}^kQ)(x)}{k!},p(x)\right)t^k,\quad N\ge m,$$
where $m={\rm deg }p$.
The theorem is proved.
\end{proof}

\begin{remark}\label{rem3}\rm The condition that $K$ has the characteristic 0 is  essentially for a uniqueness of the solution of Cauchy problem  \eqref{1}, \eqref{2} even for $a_0=0$. Indeed, let $K=\mathbb{Z}/2\mathbb{Z}$. This is a field of characteristic 2. Then the Cauchy problem  \eqref{1}, \eqref{2} for $Q(x)=0$ has a solution $u(t,x)=\sum\limits_{k=0}^{\infty}u_k(x)t^k$, where $u_0(x)=u_1(x)=0$,
$u_{2k}(x)$ is an arbitrary element of $K[x_1,...,x_n]'$ and
$u_{2k+1}(x)=({\mathcal F}u_{2k})(x)$ for any $k\in\mathbb{N}$. Therefore the considered Cauchy problem has nontrivial solutions.
\end{remark}

Example \ref{ex3} shows that the Cauchy problem \eqref{1}, \eqref{2}
may has no solutions when $a_0\not=0$. The following theorem shows
that under an additional restriction on the ring $K$ there exists a solution of this
problem even when $a_0\not=0$.

\begin{theorem}\label{th6}
For any linear differential operator
${\mathcal F}=\sum\limits_{|\alpha|=0}^{\infty}a_{\alpha}D^{\alpha}$
of infinite order
with coefficients $a_{\alpha}\in K$
and for any copolynomial  $Q\in K[x_1,...,x_n]'$ there exists a solution
of the Cauchy problem \eqref{1}, \eqref{2} has a solution
if and only if the ring $K$ contains the field of rational numbers.
Furthermore, a solution
of this Cauchy problem is unique and has the form \eqref{sol}.
\end{theorem}
\begin{proof} \textit{Sufficiency.} Since the ring $K$ contains the field of rational numbers, this ring is of characteristic 0 and $\frac{1}{k!}\in K$ for any $k\in\mathbb{N}$. Arguing as in the proof of Theorem \ref{th4}, we obtain that the module $K[x_1,...,x_n]'$
 is also a torsion-free $\mathbb{Z}$-module and
the element ${\mathcal F}^kQ$ is divided by $k!$ in the module $K[x_1,...,x_n]'$
for any $k\in\mathbb{N}$.
By Theorem 2.3 \cite{GP1}, the series \eqref{sol} is well defined,
the Cauchy problem \eqref{1}, \eqref{2} has a unique solution
and this solution has the form \eqref{sol}.

\textit{Necessity.} Suppose that for any copolynomial $Q\in K[x_1,...,x_n]'$ and for any differential operator ${\mathcal F}=\sum\limits_{|\alpha|=0}^{\infty}a_{\alpha}D^{\alpha}$ of infinite order with coefficients $a_{\alpha}\in K$ there exists a solution $u(t,x)=\sum\limits_{k=0}^{\infty}u_k(x)t^k$
of the Cauchy problem \eqref{1}, \eqref{2}. We put $a_0=1$ and $Q(x)=\delta(x)$. Then coefficients  $u_k(x)$
of the corresponding solution $u(t,x)$ satisfy equalities
$$u_0(x)=\delta(x),\quad (k+1)u_{k+1}(x)=({\mathcal F}u_k)(x),\quad k=0,1,2,....$$
Therefore
$(k+1)(u_{k+1},1)=(u_k,1)$ and $(k+1)!(u_{k+1},1)=k!(u_k,1)=1$ for any $k\in\mathbb{N}_0$.
This implies that elements $k!\in\mathbb{N}$, regarded as elements of $K$, are invertible. Therefore $\frac{1}{k!}\in K,\ k\in\mathbb{N}$.
Then $K$ contains the field of rational numbers.
\end{proof}

Theorems \ref{th4} and \ref{th6} leads to the following assertion.

\begin{corollary}\label{cor3}
Assume that one of the following two conditions is satisfied:
\begin{enumerate}
 \item The ring $K$ is of characteristic 0 and $a_0=0$.
\item The ring $K$ contains the field of rational numbers.
\end{enumerate}
 Then the Cauchy problem
$$\frac{\partial u(t,x)}{\partial t}=({\mathcal F}u)(t,x),\quad
u(0,x)=\delta(x)$$
has a unique solution in the module $K[x_1,...,x_n]'[[t]]$. This solution has the form
\begin{equation}\label{fund1}
{\mathcal E}_C(t,x)=\sum_{k=0}^{\infty}\frac{({\mathcal F}^k\delta)(x)}{k!}t^k.
\end{equation}
\end{corollary}

\begin{definition}\label{def31}\rm
 The formal power series ${\mathcal E}_C(t,x)\in K[x_1,...,x_n]'[[t]]$ defined in \eqref{fund1} is
called the \textit{fundamental solution of the Cauchy problem \eqref{1}, \eqref{2}}.
\end{definition}

 Arguing as in the proof of Theorem \ref{th6}, we obtain the following criterion  of existence of a fundamental solution of the Cauchy problem  \eqref{1}, \eqref{2} for any differential operator ${\mathcal F}=\sum\limits_{|\alpha|=0}^{\infty}a_{\alpha}D^{\alpha}$ of infinite order.

\begin{theorem}

There exists a fundamental solution of the Cauchy problem \eqref{1}, \eqref{2} for any linear differential operator
 ${\mathcal F}=\sum\limits_{|\alpha|=0}^{\infty}a_{\alpha}D^{\alpha}$
of infinite order with coefficients $a_{\alpha}\in K$
if and only if the
ring $K$ contains the field of rational numbers.
Moreover, a fundamental solution of this Cauchy problem is  unique and has the form \eqref{fund1}.
\end{theorem}

 The following assertion shows that, under the assumptions of Corollary \ref{cor3} the unique solution of the Cauchy problem \eqref{1}, \eqref{2} is represented as the convolution of the fundamental solution ${\mathcal E}_C(t,x)$ and the copolynomial $Q$.

\begin{theorem}\label{th5}
Let the assumptions of Corollary \ref{cor3} hold. Then a unique solution of the
Cauchy problem \eqref{1}, \eqref{2} can be represented in the form
$$
u(t,x)={\mathcal E}_C(t,x)*Q.
$$
\end{theorem}
\begin{proof}
Indeed, a unique solution of the Cauchy problem \eqref{1}, \eqref{2}
is defined by \eqref{sol}. On
the other hand, in view of Definition \ref{def25}
and Theorem \ref{prop1} we have
$${\mathcal E}_C(t,x)*Q=\sum_{k=0}^{\infty}\frac{({\mathcal F}^k\delta) *Q}{k!}t^k=
\sum_{k=0}^{\infty}\frac{{\mathcal F}^k(\delta*Q)}{k!}t^k=
\sum_{k=0}^{\infty}\frac{{\mathcal F}^kQ}{k!}t^k=u(t,x)$$
(see also Example \ref{ex2}).  \end{proof}

\begin{corollary}
Let the assumptions of Theorem \ref{th4} hold. Then for every fixed  $t\in K$
the sum of the series  \eqref{sol} is continuously depends on $Q$ in the topology of the module $K[x_1,...,x_n]'$.
\end{corollary}
\begin{proof}
Indeed, for every $t\in K$
 the fundamental solution ${\mathcal E}_C(t,x)$ is a copolynomial and by Theorem \ref{th5} the sum of the series \eqref{sol} can be represented as a convolution of copolynomials   ${\mathcal E}_C(t,x)$ and $Q(x)$. Now the assertion of the corollary follows from Theorem \ref{th21}.
\end{proof}
\begin{example}\label{ex31}\rm Let the ring $K$ be of characteristic 0 and let $a\in K$. In the module $K[x_1,...,x_n]'[[t]]$ we consider the heat equation
\begin{equation} \label{teplo}
\frac{\partial u(t,x)}{\partial t}=a\triangle u(t,x),\quad \triangle=\sum_{j=1}^n
\frac{\partial ^2}{\partial x_j^2}
\end{equation}
which is a particular case of Equation \eqref{1} with the differential operator
${\mathcal F}=a\triangle $. The assumptions of Theorems \ref{th4}, \ref{th5} and Corollary \ref{cor3} are satisfied.
By Theorem \ref{th4}, for any  $Q\in K[x_1,...,x_n]'$ the Cauchy problem \eqref{teplo}, \eqref{2} has a unique solution and this solution has the form
(see Formula \eqref{sol})
\begin{equation}\label{solt}
u(t,x)=\sum_{k=0}^{\infty}a^{k}\frac{\triangle ^kQ}{k!}t^k.\end{equation}
By Corollary \ref{cor3}, the fundamental solution of this Cauchy problem exists and has the form
\begin{equation}\label{fundt}
{\mathcal E}_C(t,x)=\sum_{k=0}^{\infty}a^k\frac{\triangle ^k\delta}{k!}t^k.
\end{equation}
As in Example \ref{ex1},
for any multi-index  $\beta\in\mathbb{N}_0^n$ and $k\in\mathbb{N}$ we obtain:
$$\frac{\triangle^kx^{\beta}}{k!}=\sum_{|\alpha|=k}\frac{D^{2\alpha}x^{\beta}}{\alpha!}.$$
Therefore
$$\left(\frac{\triangle^{|\alpha|}\delta}{|\alpha |!},x^{\beta}\right)=\left(\delta,\frac{\triangle^{|\alpha|}x^{\beta}}{|\alpha |!}\right)=
\left\{\begin{array}{ccc}\frac{(2\alpha)!}{\alpha!},\quad \beta=2\alpha,\\
0,\quad  \beta\not=2\alpha.\end{array}\right.$$
Substituting this expression into the formula \eqref{fundt}, we obtain the following representation for the fundamental solution of the Cauchy problem
 \eqref{teplo}, \eqref{2}:
\begin{equation}\label{fundteplo}
({\mathcal E}_C(t,x),x^{\beta}) = \left\{\begin{array}{ccc}\frac{(2\alpha)!}{\alpha!}(at)^{|\alpha|},\quad \beta=2\alpha,\\
0,\quad  \beta\not=2\alpha.\end{array}\right.\end{equation}

Now let $K=\mathbb{R}$ and $a>0$.
We show that, in the space  $\mathbb{R}[x_1,...,x_n]'$ for every $t>0$ the sum of the series \eqref{fundt}, regarded as a
regular copolynomial, has the form
$$
\sum_{k=0}^{\infty}a^k\frac{\triangle ^k\delta}{k!}t^k=\frac{1}{(\sqrt{4\pi at})^n}e^{-\frac{|x|^2}{4at}},\quad |x|^2=\sum_{j=1}^nx_j^2.$$
For this purpose we first note that
$$\frac{1}{\sqrt{4\pi at}}\int\limits_{-\infty}^{\infty } y^k e^{-\frac{y^2}{4at}}dy
=\left\{\begin{array}{ccc}\frac{(2l)!}{l!}a^lt^l,\quad k=2l, \\
0,\quad k=2l+1.\end{array}\right.
$$
Now, taking into account \eqref{fundteplo}, we obtain
$$({\mathcal E}_C(t,x),x^{\beta})=\left\{\begin{array}{ccc}\frac{(2\alpha)!}{\alpha!}(at)^{|\alpha|},\quad \beta=2\alpha,\\
0,\quad  \beta\not=2\alpha.\end{array}\right.=\frac{1}{(\sqrt{4\pi at})^n}\int\limits_{\mathbb{R}^n} x^{\beta} e^{-\frac{|x|^2}{4at}}dx.$$
\end{example}
\begin{example}\rm
Assume that the ring $K$ contains the field of rational numbers, $a\in K$ and $Q(x)\in K[x_1,...,x_n]'$. In the module $K[x_1,...,x_n]'[[t]]$ we consider the Cauchy problem for the
inhomogeneous heat equation
\begin{equation} \label{teplo1}
\frac{\partial v(t,x)}{\partial t}=a\triangle v(t,x)+Q(x),
\end{equation}
\begin{equation}\label{ic1}
v(0,x)=0.
\end{equation}
It is easy to see that if $v(t,x)\in K[x_1,...,x_n]'[[t]]$ is a solution of the Cauchy problem \eqref{teplo1}, \eqref{ic1},
then the formal power series
\begin{equation}\label{der}
u(t,x)=\frac{\partial v(t,x)}{\partial t}
\end{equation}
is a solution of the Cauchy problem
\eqref{teplo}, \eqref{2}. The unique solution of the Cauchy problem \eqref{teplo}, \eqref{2} has the form \eqref{solt} (see Example
\ref{ex31}). From \eqref{solt}, \eqref{der} and \eqref{ic1} we uniquely restore
the formal power series $v(t,x)$:
\begin{equation}\label{un}
v(t,x)=\sum_{k=0}^{\infty}a^{k}\frac{\triangle ^kQ}{(k+1)!}t^{k+1}.
\end{equation}
Since
$$\frac{\triangle ^kQ}{(k+1)!}=\sum_{|\alpha|=k}\frac{D^{2\alpha}Q}{\alpha ! (k+1)}\in K[x_1,...,x_n]',\quad k\in\mathbb{N},$$
this series is well defined.  Substituting \eqref{un} into
 \eqref{teplo1} and \eqref{ic1}, we see that the series $v(t,x)$ is a solution of the Cauchy problem \eqref{teplo1}, \eqref{ic1}. The
uniqueness of a solution of the Cauchy problem \eqref{1}, \eqref{ic1} follows from Theorem
\ref{th6}.
\end{example}

\begin{example}\rm Assume that $K$ is of characteristic 0 and $s_1,...,s_n\in K$.
We find the fundamental solution of the Cauchy problem for the transport equation
\begin{equation}\label{trans}
\frac{\partial u}{\partial t}=\sum_{j=1}^ns_j\frac{\partial u}{\partial x_j}.
\end{equation}
Equation \eqref{trans} is a particular case of Equation \eqref{1} with the differential operator ${\mathcal F}=\sum\limits_{j=1}^ns_j\frac{\partial }{\partial x_j}$.
By Corollary \ref{cor3} the fundamental solution of the Cauchy problem \eqref{trans},\eqref{2} has the form
$${\mathcal E}_C(t,x)=\sum_{k=0}^{\infty}\frac{\left(\sum\limits_{j=1}^ns_j\frac{\partial }{\partial x_j}\right)^k\delta}{k!}t^k=\sum_{k=0}^{\infty}t^k\sum_{|\alpha|=k}\frac{s^{\alpha}D^{\alpha}\delta}{\alpha!},
$$
where $s=(s_1,...,s_n)$.
Since for any  $\beta\in \mathbb{N}_0^n$
$$\left(\frac{D^{\alpha}\delta}{\alpha!},x^{\beta}\right)=\left\{\begin{array}{ccc}(-1)^{|\beta |},\quad \beta=\alpha,\\
0,\quad  \beta\not=\alpha.\end{array}\right.,$$
we obtain by virtue the definition of the shift of a copolynomial
$$({\mathcal E}_C(t,x),x^{\beta})=(-1)^{|\beta|}t^{|\beta|}s^{\beta}=(\delta(x+ts),x^{\beta}).$$
Hence the fundamental solution of the Cauchy problem for Equation \eqref{trans} coincides
with the copolynomial $\delta(x+ts)$.
\end{example}

\subsection{Connections between fundamental solutions.}\label{subs73}

Let ${\mathcal F}=\sum\limits_{|\alpha|=0}^{\infty}a_{\alpha}D^{\alpha}$ be a linear
differential operator of infinite order on $K[x_1,...,x_n]'$ with coefficients
$a_{\alpha}\in K$ and let
$a_0$ be an invertible element of the ring  $K$. We assume that the ring $K$ contains the field of rational numbers. By Corollaries \ref{cor2} and \ref{cor3}, the differential operator ${\mathcal F}$ and the Cauchy problem \eqref{1}, \eqref{2} have the fundamental solutions  ${\mathcal E}(x)$ and ${\mathcal E}_C(t,x)$ respectively.
Furthermore, by Theorem \ref{th2} the operator ${\mathcal F}$ is invertible.
 Theorem \ref{th2} and Corollary \ref{cor2} imply that the differential operator $\frac{\partial }{\partial t}-{\mathcal F}:
K[t,x_1,...,x_n]'\to K[t,x_1,...,x_n]'$ is also invertible and this operator also has a fundamental solution which will be denoted by $\tilde{{\mathcal E}}(t,x)$.

At first, we give the connections between fundamental solutions ${\mathcal E}(x)$ and ${\mathcal E}_C(t,x)$. By definitions of the
fundamental solutions of an operator and a Cauchy problem
(see equalities \eqref{fund}) and \eqref{fund1}), we obtain the formula $${\mathcal E}(x)=({\mathcal F}^{-1}{\mathcal E}_C)(0,x),$$
which expresses the fundamental solution of the operator ${\mathcal F}$ through the fundamental solution of the Cauchy problem \eqref{1}, \eqref{2}.
With the help of \eqref{fund} and \eqref{fund1} we obtain the formula
$${\mathcal E}_C(t,x)=\sum_{k=0}^{\infty}\frac{({\mathcal F}^{k+1}{\mathcal E})(x)}{k!}t^k,$$
which expresses the fundamental solution of the Cauchy problem \eqref{1}, \eqref{2} through
the fundamental solution of the operator ${\mathcal F}$.

Now we establishes connections between fundamental solutions $\tilde{{\mathcal E}}(t,x)$
and ${\mathcal E}(x)$. We have
$$\frac{\partial }{\partial t}-{\mathcal F}={\mathcal F}\left({\mathcal F}^{-1}\frac{\partial }{\partial t}-I\right).$$
Then the operator $\left({\mathcal F}^{-1}\frac{\partial }{\partial t}-I\right):K[t,x_1,...,x_n]'\to K[t,x_1,...,x_n]'$ is invertible and the operator equality
\begin{equation}\label{t1}
\left(\frac{\partial }{\partial t}-{\mathcal F}\right)^{-1}=
\left({\mathcal F}^{-1}\frac{\partial }{\partial t}-I\right)^{-1}{\mathcal F}^{-1}
\end{equation}
holds. Applying the equality  \eqref{t1} to the copolynomial $\delta(t,x)=\delta(t)\otimes \delta(x)\in K[t,x_1,...,x_n]'$, we obtain the formulas
$$\tilde{{\mathcal E}}(t,x)=\left({\mathcal F}^{-1}\frac{\partial }{\partial t}-I\right)^{-1}{\mathcal F}^{-1}(\delta(t)\otimes \delta(x))=
\left({\mathcal F}^{-1}\frac{\partial }{\partial t}-I\right)^{-1}(\delta(t)\otimes ({\mathcal F}^{-1}\delta)(x))$$
$$=\left({\mathcal F}^{-1}\frac{\partial }{\partial t}-I\right)^{-1}(\delta(t)\otimes{\mathcal E}(x))=\left({\mathcal F}^{-1}\frac{\partial }{\partial t}-I\right)^{-1}(\delta(t)\otimes({\mathcal F}^{-1}{\mathcal E}_C)(0,x)),$$
which express the fundamental solution of the operator
$\frac{\partial }{\partial t}-{\mathcal F}$ through  either the fundamental solution of the operator ${\mathcal F}$ or the fundamental solution of the Cauchy problem \eqref{1}, \eqref{2}.
Moreover, we obtain the formula
$$\delta(t)\otimes{\mathcal E}(x)=\left({\mathcal F}^{-1}\frac{\partial }{\partial t}-I\right)\tilde{{\mathcal E}}(t,x),$$
which implicitly expresses the fundamental solution of the operator ${\mathcal F}$ through the
fundamental solution of the operator
$\frac{\partial }{\partial t}-{\mathcal F}$.

\subsection*{Acknowledgments.} This work was supported by the Akhiezer Foundation. The authors grateful to Eugene Karolinsky and Sergey Favorov for
useful discussions of the paper results.


\begin{thebibliography}{99}


\bibitem{burb} N. Bourbaki, \emph{Elements de Mathematique. Premiere Partie: Les Structures Fondamentales de l’Analyse. Livre II: Algebre. Chap. II},
Hermann, Paris, 1962 (French).
\bibitem{Dub} Yu. A. Dubinskii, \emph{Cauchy Problem in a Complex Domain}, Moscow Energ. Inst.
Press, Moscow, 1996 (Russian).
\bibitem{EK}
R. Estrada, R.P. Kanwal,
\emph{ A distributional approach to asymptotics theory and applications},
{\rm Birkh\"{a}user}, 2002.
\bibitem {v1} I. Frenkel, J. Lepovsky, and A. Meurman, \emph{Vertex Operator Algebras and the Monster}, Academic Press, New York, 1988.
\bibitem{G} S.L. Gefter, \emph{Differential Operators of Infinite Order in the Space of Formal Laurent Series and in the Ring of Power Series with Integer Coefficients}, J. Math. Sci. \textbf{239} (2019), No. 3,  282-–291 .
\bibitem{GP1} S.L. Gefter, A.L. Piven’, \emph{Linear Partial Differential Equations in Module of Formal Generalized Functions over Commutative Ring}, J. Math. Sci., \textbf{257} (2021), No.5, 579--596.
\bibitem{GP} S.L. Gefter, A.L. Piven’, \emph{Implicit Linear Differential-Difference Equations in the Module of Formal Generalized Functions over a Commutative Ring}, J. Math. Sci., \textbf{255} (2021), No. 4, 409--422.
\bibitem{GS} S.L. Gefter, T.E. Stulova, \emph{Fundamental Solution of the Simplest Implicit Linear Differential Equation in a Vector Space}, J. Math. Sci., \textbf{207} (2015), No.2, 166--175.
\bibitem{GV} S. Gefter, A. Vershynina,   \emph{On analytic solutions of the heat equation with an operator coefficient}, J. Math. Sci., \textbf{156} (2009), No.5, 799--812.
\bibitem{Gorod1} V. V. Gorodetskii and R. S. Kolisnyk, \emph{Cauchy problem for evolution equations with an
infinite-order differential operator. I},  Differ. Equ. \textbf{43} (2007), No. 8, 1111--1122.
\bibitem{Gorod} V. V. Gorodetskii and R. S. Kolisnyk, \emph{Cauchy problem for evolution equations with an
infinite-order differential operator. II},  Differ. Equ. \textbf{43} (2007), No. 9, 1181--1193.
\bibitem{Gr} H. Grauert and R. Remmert, \emph{Analytische Stellenalgebren}, Springer, Berlin, 1971.
\bibitem{GGG} S.L. Hefter, A.B. Goncharuk,  \emph{Linear Differential Equation with Inhomogeneity in the Form of a Formal Power Series Over a Ring with Non-Archimedean Valuation}. Ukr. Math. J. \textbf{74} (2023), No. 11, 1668--1685.
\bibitem{HP} S.L. Hefter, O.L. Piven’,  \emph{Infinite-Order Differential Operators in the Module of Formal Generalized Functions and in a Ring of Formal Power Series}. Ukr. Math. J. \textbf{74} (2022), No. 6,  896--915.
\bibitem{EF} L.G. Hern\'andez, R. Estrada,  \emph{Solutions of Ordinary Differential Equations by Series of Delta Functions}, J. Math. Anal. Appls. \textbf{191} (1995),  40 -- 55.
 \bibitem{Ho1}   L.  H\"ormander, \emph{The Analysis of Linear Partial Differential Equations. 1. Distribution
Theory and Fourier Analysis}, Springer-Verlag, Berlin, 1983.
\bibitem{Ho2} L. H\"ormander, \emph{The Analysis of Linear Partial Differential Operators II. Differential Operators with Constant Coefficients}, Springer-Verlag, Berlin, 1983.
\bibitem{v2} V. G. Kac, \emph{Vertex Algebras for Beginners}, Am. Math. Soc., Providence, RI, 1998.
\bibitem{kn} A. S. Krivosheev, V. V. Napalkov, \emph{Complex analysis
and convolution operators}, Russian Mathematical Surveys,
\textbf{47} (1992), No. 6, 1--56.
\bibitem{Kur} V.L. Kurakin, \emph{Hopf Algebras of Linear Recurring Sequences over Rings and Modules}, J. Math. Sci. \textbf{128} (2005), No. 6, 3402--3427.
\bibitem{LN} R. Lidl,  H. Niederreiter, \emph{Finite Fields}, Cambridge University Press, 1996.
\bibitem{Mor} M. Morimoto, \emph{An Introductions to Sato's Hyperfunctions}, AMS Providence, Rhode Island, 1993.
\bibitem{St1} S. Steinberg, F. Treves, \emph{Pseudo-Fokker Planck Equations
and Hyperdifferential Operators}, J. Diff. Eq. \textbf{8}  (1970), 333--366.
\bibitem{St} S. Steinberg, \emph{The Cauchy Problem for Differential Equations
of Infinite Order}, J. Diff. Eq. \textbf{9} (1971), 591--607.
\bibitem{Tkach} V.A. Tkachenko, \emph{Spectral theory in spaces of analytic functionals
for operators generated by multiplication by the independent
variable}, Mathematics of the USSR-Sbornik, \textbf{40} (1981),
No. 3, 387--427.
\bibitem{Und} R.G. Underwood, \emph{Fundamentals of Hoph algebras}, Springer, Universitext, 2015.



\end{thebibliography}
\end{document}